\pdfoutput=1
\documentclass[11pt]{amsart}

\usepackage[a4paper,margin=1in]{geometry}
\usepackage{amsmath,amssymb,amsthm,mathtools}
\usepackage{enumitem}
\usepackage[hidelinks]{hyperref}
\usepackage{cleveref}
\usepackage{microtype}
\usepackage{booktabs}
\usepackage{array}
\usepackage{tikz}
\usetikzlibrary{arrows,positioning,calc,fit,shapes.geometric}

\title{Popular Differences and the Croot--Lev Half-Threshold Problem}

\author[J. Hou]{Jianfeng Hou}
\address{Fuzhou University, Fuzhou, Fujian, China}
\email{jfhou@fzu.edu.cn}

\author[W. Li]{Wei Li}
\address{Fujian Agriculture and Forestry University, Fuzhou, Fujian, China}
\email{liwei@fafu.edu.cn}

\author[K. Yang]{Kai Yang}
\address{Fuzhou University, Fuzhou, Fujian, China}
\email{443926471@qq.com}
\thanks{Corresponding author: Kai Yang.}

\date{}

\newtheorem{theorem}{Theorem}[section]
\newtheorem{lemma}[theorem]{Lemma}
\newtheorem{proposition}[theorem]{Proposition}
\newtheorem{corollary}[theorem]{Corollary}
\newtheorem{problem}[theorem]{Problem}
\theoremstyle{definition}
\newtheorem{definition}[theorem]{Definition}
\newtheorem{example}[theorem]{Example}
\theoremstyle{remark}
\newtheorem{remark}[theorem]{Remark}

\DeclareMathOperator{\Stab}{Stab}

\newcommand{\F}{\mathbb F}

\begin{document}

\begin{abstract}
Let $A$ be a finite non-empty subset of an abelian group $G$, and let
$r_A(d)=|\{(a,a')\in A^2:a-a'=d\}|$. Croot and Lev asked whether the
pointwise half-threshold condition
$r_A(d)\ge |A|/2$ for every $d\in A-A$ forces $A-A$ to be either a
subgroup or a union of three cosets. We resolve this open problem in its
sharp general form by identifying the essential obstruction: the statement
is false in arbitrary abelian groups, but becomes true after excluding
non-zero two-torsion. More precisely, if $G$ is two-torsion-free and the
half-threshold condition holds, then either $A-A$ is a finite subgroup of
$G$, or there are a finite subgroup $H\le G$ and elements $x,g\in G$ such
that
\[
    A=(x+H)\cup(x+g+H).
\]
The two-torsion-free hypothesis is
essential: for every $r\ge1$ we construct
$A\subseteq\F_2^{2r+1}$ with
$A-A=\F_2^{2r+1}\setminus\{t\}$ such that every non-zero represented
difference has exactly $|A|/2$ representations, giving genuine
counterexamples to the Croot--Lev conclusion. The proof of the positive
result combines a Kneser quotient reduction with Lev's formulation of
Kemperman's critical-pair theory.
\end{abstract}

\maketitle


\section{Introduction}

Let $G$ be an abelian group, and let $A$ be a finite non-empty subset of $G$. Write
\[
G[2]=\{g\in G:2g=0\}.
\]
Thus $G[2]=\{0\}$ means that $G$ has no non-zero element of order two; in this case, we say that $G$ is \emph{$2$-torsion-free}. Let  $D:= A-A=\{a-a':a,a'\in A\}$ be \emph{difference set} of $A$. The \emph{difference representation function}~\cite{Nathanson,NathansonRep} $r_A$ of $A$ is defined on $D$ by 
\[
    r_A(d)=\left|\{(a,a')\in A^2:a-a'=d\}\right|, \qquad d\in D.
\]
Clearly, \(r_A(d)\) measures exactly how much the translate \(A+d\) overlaps
with \(A\). The first
two global pieces of information carried by \(r_A\) are
\[
    \sum_{d\in G} r_A(d)=|A|^2,
    \qquad
    E(A)=\sum_{d\in G}r_A(d)^2.
\]
The first identity gives the averaging relation between the size of
\(D\) and the multiplicities of its differences. The second quantity is
the additive energy, which is one of the standard signals of additive structure \cite{BalogSzemeredi,Gowers,GreenRuzsa,RuzsaSurvey,ShkredovHigherEnergies}

When studying whether sumsets or difference sets contain large linear subspaces~\cite{GreenFiniteField}, one often needs to consider the ``popular'' elements $d\in D$, namely those for which $r_A(d)$ is large. Wolf~\cite{WolfPopular} studied the structure of popular difference sets in this direction, while Sanders~\cite{SandersPopular} explained the relevance of such constructions to the problem of finding large subspaces in sumsets. Popular differences are also closed related to generalized Sidon sets \cite{KonyaginLevPopular,XuPopular}.

In this paper, we concern an extremal inverse problem in which every represented difference is assumed to be popular. Croot and Lev formulated the following question in their open-problem collection \cite[Problem~7.14]{CrootLev}.

\begin{problem}[Croot--Lev]\label{prob:croot-lev}
Let $A$ be a finite non-empty subset of an abelian group $G$. Suppose that
\begin{equation}\label{eq:half-popular}
    r_A(d)\ge \frac{|A|}{2}\qquad\text{for every }d\in D
\end{equation}
holds. Must $D$ be either a subgroup of $G$ or a union of three cosets
of a subgroup of $G$?
\end{problem}

In the same paper, Croot and Lev observed that the strict supercritical version has a simple answer: if $r_A(d)>|A|/2$ for every $d\in D$, then $D$ is a subgroup. 
However, Problem~\ref{prob:croot-lev} is not true for general groups; we provide counterexamples in Section~\ref{sec:counterexamples-remarks}. 
On the other hand, we observe that the $2$-torsion-free assumption is indispensable. 
Indeed, Problem~\ref{prob:croot-lev} holds for all $2$-torsion-free abelian groups.

\begin{theorem}\label{thm:main}
Let $G$ be a $2$-torsion-free abelian group, and let $A\subseteq G$ be
finite and non-empty.  Suppose that
\[
    r_A(d)\ge \frac{|A|}{2}\qquad\text{for every }d\in D.
\]
Then either $D$ is a finite subgroup of $G$, or there exist a finite
subgroup $H\le G$ and elements $x,g\in G$ such that
\begin{equation}\label{eq:A-two-cosets}
    A=(x+H)\cup(x+g+H),
\end{equation}
and the image of $g$ in $G/H$ has order different from $1$, $2$, and $3$.
In the second case
\begin{equation}\label{eq:D-three-cosets}
    D=(-g+H)\cup H\cup(g+H). 
\end{equation}
\end{theorem}

For finite abelian groups, $2$-torsion-free is equivalent to having odd order. Thus \Cref{thm:main} gives the following direct answer to the Problem~\ref{prob:croot-lev} in odd-order groups.

\begin{corollary}\label{cor:odd-order}
Let $G$ be a finite abelian group of odd order, and let $A\subseteq G$ be
non-empty. If \eqref{eq:half-popular} holds, then either $A-A$ is a
subgroup of $G$, or $A$ is the union of two cosets of a subgroup $H\le G$.
In the latter case $A-A$ is precisely the three-coset set
\eqref{eq:D-three-cosets}; the image of $g$ in $G/H$ has odd order at
least five.
\end{corollary}

The remainder of this paper is organized as follows. We collect some useful lemmas in \Cref{sec:preliminaries}, and prove Theorem \ref{thm:main} in \Cref{sec:proof-main}. The characteristic-two counterexamples and the final
remarks are collected in \Cref{sec:counterexamples-remarks}.

\section{Preliminaries}\label{sec:preliminaries}

We begin with two elementary observations. The first is the strict version
mentioned by Croot and Lev~\cite[Problem~7.14 and the following
discussion]{CrootLev}.

\begin{proposition}\label{prop:strict}
Let $A$ be a finite non-empty subset of an abelian group $G$. If
\[
    r_A(d)>\frac{|A|}{2}\qquad\text{for every }d\in D,
\]
then $D$ is a subgroup of $G$.
\end{proposition}

\begin{proof}
It is enough to show that $D$ is closed under subtraction, since $0\in D$
and $D=-D$. Let $x,y\in D$, and define
\[
    T_x=\{a\in A:a+x\in A\},\qquad
    T_y=\{a\in A:a+y\in A\}.
\]
Then $|T_x|=r_A(x)>|A|/2$ and $|T_y|=r_A(y)>|A|/2$. Hence
$T_x\cap T_y\ne\emptyset$. Choose $a\in T_x\cap T_y$. Then
$a+x,a+y\in A$, and so 
\[
    x-y=(a+x)-(a+y)\in D.
\]
Thus $D-D\subseteq D$, and $D$ is a subgroup.
\end{proof}

\begin{lemma}\label{lem:counting}
Let $A$ be a finite non-empty subset of an abelian group, 
and suppose that \eqref{eq:half-popular} holds. Then
\[
    |D|\le 2|A|-1.
\]
\end{lemma}

\begin{proof}
Let $n=|A|$. Since $r_A(0)=n$ and
$\sum_{d\in D}r_A(d)=n^2$, the hypothesis gives
\[
    n^2
    =\sum_{d\in D}r_A(d)
    \ge n+(|D|-1)\frac{n}{2}.
\]
Dividing by $n/2$ gives $|D|\le 2n-1$.
\end{proof}

The next lemma is an elementary group-theoretic observation, whose proof is included here
for completeness.
\begin{lemma}\label{lem:no-two-quotient}
Let $G$ be a $2$-torsion-free abelian group, and let $L\le G$ be a
finite subgroup. Then both $L$  and $(G/L)$ are $2$-torsion-free.
\end{lemma}

\begin{proof}
Clearly, $L$ is $2$-torsion-free. If $x+L\in G/L$ satisfies $2(x+L)=0$, then $2x\in L$.
Choose $\ell\in L$ with $2\ell=2x$. By the $2$-torsion-freeness of $L$, we have $x=\ell$ and so  $x+L=0$.
\end{proof}

We shall use Kneser's theorem in the following form (See Kneser \cite{Kneser} or, for a modern reference, Tao and Vu \cite[Chapter~5]{TaoVu}). Let  $X,Y$ be finite non-empty subsets of an abelian group. If
\[
\begin{aligned}
    X+Y&=\{x+y: x\in X, y\in Y\},\\
    K&=\Stab(X+Y)=\{g:X+Y+g=X+Y\}.
\end{aligned}
\]
then
\begin{equation}\label{eq:kneser}
    |X+Y|\ge |X+K|+|Y+K|-|K|.
\end{equation}
Let 
\[
    \mu(X,Y)=\min_{z\in X+Y}|\{(x,y)\in X\times Y:z=x+y\}|.
\]

Kemperman's structure theorem  \cite{Kemperman} classifies critical pairs satisfying
$|X+Y|\le |X|+|Y|-1$. We recall the form relevant here,
following Lev's exposition \cite[Section~3]{LevKemperman}. For other
modern treatments and proofs, see Grynkiewicz~\cite{Grynkiewicz} and
Boothby, DeVos, and Montejano~\cite{BoothbyDeVosMontejano}. For completeness we recall only the parts of the elementary-pair
classification needed below.

\begin{definition}\label{def:elementary-pairs}
A pair $(X,Y)$ of finite non-empty subsets of an abelian group $Q$ is
\emph{elementary} if one of the four Kemperman types holds. 
\begin{enumerate}[label=(\Roman*)]
\item One of $X$ and $Y$ has size one.
\item Both $X$ and $Y$ are arithmetic progressions with a common
difference whose order is at least $|X|+|Y|-1$.
\item There are elements $x_0,y_0\in Q$ and non-empty subsets
$H_1,H_2$ of a finite subgroup $H\le Q$ such that
\[
    H=H_1\cup H_2\cup\{0\},\qquad
    X=x_0+(H_1\cup\{0\}),\qquad
    Y=y_0-(H_2\cup\{0\}),
\]
and $x_0+y_0$ is the unique element of $X+Y$ with exactly one
representation.
\item There are elements $x_0,y_0\in Q$ and non-empty aperiodic subsets
$H_1,H_2$ of a finite subgroup $H\le Q$ such that
\[
    H=H_1\cup H_2,\qquad
    X=x_0+H_1,\qquad
    Y=y_0-H_2,
\]
and $\mu(X,Y)\ge2$.
\end{enumerate}
\end{definition}

The only Kemperman-theoretic input used below is the following light
formulation of Lev \cite[Theorem~1, p.~385]{LevKemperman}.

\begin{theorem}\label{thm:lev-reduction}
Let $X,Y$ be finite non-empty subsets of a non-trivial abelian group $Q$
such that
\[
    |X+Y|\le |X|+|Y|-1,
\]
and suppose that either $X+Y\ne Q$ or $\mu(X,Y)=1$. Then there is a finite
proper subgroup $L<Q$ such that, writing $\rho:Q\to Q/L$ for the quotient
map, the pair $(\rho(X),\rho(Y))$ is elementary.
\end{theorem}

For a self-opposite pair $(C,-C)$ the following
consequence is all that will be used.

\begin{lemma}\label{lem:elementary-consequence}
Let $C$ be a finite non-empty subset of an abelian group $Q$, and suppose
that $(C,-C)$ is elementary. If $|C|\ge2$, then either $C-C$ contains a
non-zero element with a unique representation as $c-c'$ with $c,c'\in C$,
or the elementary pair is of type IV; in the latter case $Q$ contains a
finite subgroup of order $2|C|$.
\end{lemma}

\begin{proof}
Type I is excluded as $|C|\ge2$. In type II, the two sets are arithmetic
progressions with a common difference $d$ whose order is at least
$2|C|-1$. An endpoint difference is represented uniquely and is non-zero.
In type III, the elementary-pair definition gives an element of $C-C$ with
a unique representation; it cannot be $0$, because $0=c-c$ has $|C|\ge2$
representations. In type IV, the defining subgroup has order
$|C|+|-C|=2|C|$.
\end{proof}

\section{Proof of Theorem \ref{thm:main}}\label{sec:proof-main}

In this section, we prove Theorem \ref{thm:main}:  The argument begins with the elementary counting consequence of the half-popularity assumption: the difference set is small enough to lie in Kneser's critical range. We then quotient by the stabilizer of the difference set and reduce the problem to an aperiodic critical pair. The key step is that the original pointwise lower bound becomes a uniform representation property in this quotient, which is rigid enough to be combined with Lev's Kemperman-type reduction. In the two-torsion-free setting this leaves only a two-point quotient; lifting back through the fibres gives two complete cosets, while the characteristic-two
construction explains why the torsion hypothesis is necessary.

\subsection{Kneser reduction}\label{sec:kneser}

The first structural step is independent of parity.

\begin{theorem}\label{thm:kneser-reduction}
Let $A$ be a finite non-empty subset of an abelian group $G$,  and suppose that
\[
    r_A(d)\ge \frac{|A|}{2}\qquad\text{for every }d\in D.
\]
Let
\[
    H=\Stab(D),
\]
let $\pi:G\to G/H$ be the quotient map, and set $B=\pi(A)$. Write
$h=|H|$ and $a=|B|$. Then $B-B$ is aperiodic in $G/H$, and the following
statements hold.
\begin{enumerate}[label=\textup{(\roman*)}]
    \item $|B-B|=2a-1$.
    \item If $\delta=|A+H|-|A|=ah-|A|$ is the total number of holes of
    $A$ inside the $H$-cosets it meets, then
    \[
        \delta\le \frac{h-1}{2}.
    \]
    \item If $a>1$, then $a$ is even and
    \[
        r_B(\xi)=\frac{a}{2}
        \qquad\text{for every }0\ne\xi\in B-B.
    \]
\end{enumerate}
Consequently, $a=1$ gives the subgroup case $D=H$, while $a=2$ gives
\[
    D=(-g+H)\cup H\cup(g+H)
\]
for some $g\in G$.
\end{theorem}

\begin{proof}
Since $D$ is finite, its stabilizer $H$ is finite. Since $H=\Stab(D)$, the
set $D$ is $H$-periodic and
\[
    D=\pi^{-1}(B-B).
\]
The quotient difference set $B-B$ has trivial stabilizer in $G/H$. Indeed,
if $\tau\in G/H$ stabilizes $B-B$ and $t\in G$ represents $\tau$, then
the identity $D=\pi^{-1}(B-B)$ gives $D+t\subseteq D$. Applying the same
argument to $-\tau$ gives the reverse inclusion, so $D+t=D$. Hence
$t\in\Stab(D)=H$, and so $\tau=0$.
In particular,
\begin{equation}\label{eq:D-quotient-size}
    |D|=h|B-B|.
\end{equation}
By Kneser's theorem applied to $A+(-A)$ with stabilizer $H$, we have
\[
    |D|\ge |A+H|+|-A+H|-|H|=2ah-h.
\]
Hence
\begin{equation}\label{eq:lower-Bdiff}
    |B-B|\ge 2a-1.
\end{equation}
On the other hand, \Cref{lem:counting} gives
\[
    |D|\le 2|A|-1\le 2ah-1.
\]
Together with \eqref{eq:D-quotient-size}, this yields
\[
    |B-B|\le 2a-1.
\]
Combining this with \eqref{eq:lower-Bdiff} proves (i).

Using $|D|=h(2a-1)$ and $|A|=ah-\delta$ in the inequality
$|D|\le2|A|-1$, we get
\[
    h(2a-1)\le 2(ah-\delta)-1.
\]
This is equivalent to $2\delta\le h-1$, proving (ii).

It remains to prove (iii). Fix $\xi\in B-B$. Since $D=\pi^{-1}(B-B)$,
every element $d\in \pi^{-1}(\xi)$ lies in $D$. For such a $d$, every
representation $d=x-y$ with $x,y\in A$ projects to a representation
$\xi=b-b'$ with $b,b'\in B$. For each fixed quotient representation
$b-b'=\xi$, there are at most $h$ lifts $(x,y)\in A^2$ with $x-y=d$.
Thus
\begin{equation}\label{eq:lift-bound}
    r_A(d)\le h\,r_B(\xi).
\end{equation}
By the popularity hypothesis and (ii),
\[
    r_B(\xi)
    \ge \frac{|A|}{2h}
    =\frac{ah-\delta}{2h}
    \ge \frac{a}{2}-\frac{h-1}{4h}
    >\frac{a}{2}-\frac14.
\]
Since $r_B(\xi)$ is an integer, this implies
\begin{equation}\label{eq:quotient-lower}
    r_B(\xi)\ge \left\lceil\frac{a}{2}\right\rceil
    \qquad(\xi\in B-B).
\end{equation}
For the zero quotient difference, $r_B(0)=a$. For the non-zero quotient
differences, we use (i): there are $2a-2$ of them, and
\[
    \sum_{0\ne \xi\in B-B} r_B(\xi)=a^2-a.
\]
The average value of $r_B(\xi)$ over $0\ne\xi\in B-B$ is therefore
\[
    \frac{a^2-a}{2a-2}=\frac a2.
\]
If $a>1$ were odd, \eqref{eq:quotient-lower} would force every non-zero
$r_B(\xi)$ to be at least $(a+1)/2$, contradicting this average. Hence
$a$ is even. For even $a$, the lower bound \eqref{eq:quotient-lower}
coincides with the average, so equality must hold for every non-zero
$\xi\in B-B$. This proves (iii).

Finally, if $a=1$, then $B-B=\{0\}$ and $D=H$. If $a=2$, say
$B=\{\pi(x),\pi(y)\}$, then
\[
    B-B=\{0,\pi(x-y),\pi(y-x)\}.
\]
Choosing $g=x-y$, we obtain
\[
    D=(-g+H)\cup H\cup(g+H).
\]
\end{proof}

We next record a strengthening of the last case of \Cref{thm:kneser-reduction}.
It is this observation that upgrades the odd-order result from a
classification of $D$ to a classification of $A$ itself.

\begin{proposition}\label{prop:two-fibres-full}
In the setting of \Cref{thm:kneser-reduction}, suppose that $|\pi(A)|=2$.
Then there exist $x,g\in G$ such that
\[
    A=(x+H)\cup(x+g+H).
\]
Consequently, in this case $A$ is exactly the union of two full
$H$-cosets.
\end{proposition}

\begin{proof}
Translate $A$ so that
\[
    A=X\cup(g+Y),
    \qquad X,Y\subseteq H,
\]
where $X$ and $Y$ are non-empty. Since $|\pi(A)|=2$, \Cref{thm:kneser-reduction} gives
$|\pi(A)-\pi(A)|=3$. Hence the two quotient differences
$\pm(g+H)$ are non-zero and distinct, and
\[
    g+H\subseteq A-A.
\]
For $h\in H$, the representations of $g+h$ as a difference of two elements
of $A$ are precisely the representations
\[
    g+h=(g+y)-x,
    \qquad y\in Y,\ x\in X.
\]
Thus
\[
    r_A(g+h)=r_{Y,X}(h)\le \min\{|X|,|Y|\}.
\]
The popularity condition gives
\[
    r_{Y,X}(h)\ge \frac{|A|}{2}=\frac{|X|+|Y|}{2}
    \qquad(h\in H).
\]
Therefore $|X|=|Y|$ and $r_{Y,X}(h)=|X|$ for every $h\in H$. Taking
$h=0$ gives $Y=X$. For arbitrary $h\in H$, the equality
$r_{X,X}(h)=|X|$ implies $X+h=X$. Hence every $h\in H$ stabilizes $X$, and
so $X=H$. Similarly $Y=H$. This proves the claim.
\end{proof}

\subsection{Uniform critical quotients}\label{sec:uniform}

The Kneser reduction leaves only one possible obstruction: a large
aperiodic quotient $B$ whose non-zero differences are all represented
exactly $|B|/2$ times. We first rule this out in finite odd-order groups.

\begin{lemma}\label{lem:odd-uniform}
Let $Q$ be a finite abelian group of odd order, and let $B\subseteq Q$ be
non-empty. Put $E=B-B$. Suppose that
\begin{equation}\label{eq:uniform-critical}
    \Stab(E)=\{0\},\qquad |E|=2|B|-1,
\end{equation}
and, if $|B|>1$, then
\begin{equation}\label{eq:uniform-rep}
    r_B(\xi)=\frac{|B|}{2}
    \qquad\text{for every }0\ne\xi\in E.
\end{equation}
Then $|B|\le2$.
\end{lemma}

\begin{proof}
Assume, for contradiction, that the lemma fails, and choose a counterexample
$(Q,B)$ with $|Q|$ minimal. Put $a=|B|$. Then $a\ge4$ and $a$ is even by
\eqref{eq:uniform-rep}. Since $E$ is aperiodic and $Q\ne\{0\}$, we have
$E\ne Q$.

Apply \Cref{thm:lev-reduction} to the critical pair $(B,-B)$. Thus there
is a finite proper subgroup $L<Q$ such that, with $\rho:Q\to Q/L$ and
$C=\rho(B)$, the pair $(C,-C)$ is elementary.

If $|C|=1$, then $B$ is contained in one coset of $L$. Translating $B$, we
may regard it as a subset of the smaller odd-order group $L$. The set
$E=B-B$ is then a subset of $L$, and its stabilizer inside $L$ remains
trivial. Hence $(L,B)$ satisfies \eqref{eq:uniform-critical} and
\eqref{eq:uniform-rep}, contradicting the minimality of $|Q|$.

Suppose next that $C-C$ has a non-zero element $\eta$ with a unique
quotient representation, say
\[
    \eta=c_1-c_2,
    \qquad c_1,c_2\in C.
\]
Let
\[
    X=B\cap \rho^{-1}(c_1),\qquad
    Y=B\cap \rho^{-1}(c_2).
\]
For every $d\in X-Y$, all representations of $d$ as a difference of two
elements of $B$ project to the unique quotient representation
$c_1-c_2=\eta$. Therefore
\[
    r_B(d)=r_{X,Y}(d)\le \min\{|X|,|Y|\}.
\]
Since $\eta\ne0$, every $d\in X-Y$ is non-zero, and \eqref{eq:uniform-rep}
gives $r_B(d)=a/2$. Hence
\[
    |X|\ge a/2,
    \qquad
    |Y|\ge a/2.
\]
As $X$ and $Y$ are disjoint subsets of $B$ and $|B|=a$, it follows that
\[
    |X|=|Y|=a/2
\]
and that $B=X\cup Y$.

Choose $d_0\in X-Y$. Since $r_{X,Y}(d_0)=|X|=|Y|$, we have
$X=Y+d_0$. Thus $Y=X-d_0$. Now let $s\in X-X$. For every
representation $s=x_1-x_2$ with $x_1,x_2\in X$, the element
$x_1-d_0$ lies in $Y$ and
\[
    (x_1-d_0)-x_2=s-d_0.
\]
Conversely, every representation of $s-d_0$ as $y-x$ with
$y\in Y$ and $x\in X$ arises in this way. Hence
\[
    r_{Y,X}(s-d_0)=r_X(s).
\]
The quotient of $s-d_0$ is $c_2-c_1=-\eta$, so $s-d_0\ne0$. Moreover,
by the uniqueness of the quotient representation $c_2-c_1$ (obtained by negating the unique representation of $\eta$), all
representations of $s-d_0$ as a difference of two elements of $B$ come
from $Y-X$. Therefore \eqref{eq:uniform-rep} gives
\[
    r_X(s)=r_{Y,X}(s-d_0)=r_B(s-d_0)=a/2=|X|.
\]
Thus $X+s=X$ for every $s\in X-X$. Consequently
$K=X-X$ is a subgroup of $Q$, and $X$ is a coset of $K$. The set
$B=X\cup Y$ is therefore a union of two cosets of $K$, and
\[
    E=B-B=K\cup(d_0+K)\cup(-d_0+K).
\]
This makes $E$ $K$-periodic. Since $\Stab(E)=\{0\}$, we get $K=\{0\}$,
and hence $|X|=1$. Therefore $a=2$, contradicting $a\ge4$.

It remains only to inspect the elementary possibilities for $(C,-C)$. By
\Cref{lem:elementary-consequence}, either the already treated unique
quotient difference exists, or the pair is of type IV. In the type IV case
the quotient group $Q/L$ contains a subgroup of order $2|C|$, which is
even. This is impossible because $Q/L$ has odd order. The contradiction
proves the lemma.
\end{proof}

We shall also use the following extension, which is needed for the full
form of \Cref{thm:main}.

\begin{lemma}\label{lem:notwo-uniform}
Let $Q$ be an abelian group with $Q[2]=\{0\}$, and let $B\subseteq Q$ be
finite and non-empty. Put $E=B-B$. Suppose that \eqref{eq:uniform-critical}
and \eqref{eq:uniform-rep} hold. Then $|B|\le2$.
\end{lemma}

\begin{proof}
Assume $|B|\ge4$. Since $E$ is finite and aperiodic, $E\ne Q$ unless $Q$
is trivial, which is impossible. Apply \Cref{thm:lev-reduction} to
$(B,-B)$. We obtain a finite proper subgroup $L<Q$ such that, with
$\rho:Q\to Q/L$ and $C=\rho(B)$, the pair $(C,-C)$ is elementary.

If $|C|=1$, then, after translation, $B\subseteq L$. The finite subgroup
$L$ has odd order by \Cref{lem:no-two-quotient}. This contradicts
\Cref{lem:odd-uniform} applied inside $L$.

If $C-C$ has a non-zero uniquely represented quotient difference, the
fibre argument in the proof of \Cref{lem:odd-uniform} gives $|B|=2$, a
contradiction. Thus, by \Cref{lem:elementary-consequence}, only type IV
could remain.

By \Cref{lem:no-two-quotient}, the quotient $Q/L$ has no element of order
two. But in type IV the quotient group contains a finite subgroup of order
$2|C|$, and such a subgroup has an element of order two. This
contradiction proves the lemma.
\end{proof}

\subsection{Completion of the proof}\label{sec:completion}

\begin{proof}[Proof of \Cref{thm:main}]
Let
\[
    H=\Stab(D),
    \qquad \pi:G\to G/H,
    \qquad B=\pi(A).
\]
By \Cref{thm:kneser-reduction}, if $|B|=1$, then $D=H$ is a finite
subgroup. We may therefore assume that $D$ is not a subgroup. Then
$|B|>1$ and
\[
    \Stab(B-B)=\{0\},
    \qquad |B-B|=2|B|-1,
\]
with
\[
    r_B(\xi)=\frac{|B|}{2}
    \qquad(0\ne\xi\in B-B).
\]
Since $H$ is finite and $G[2]=\{0\}$, \Cref{lem:no-two-quotient} shows
that the quotient $G/H$ also has no element of order two. Hence
\Cref{lem:notwo-uniform} gives $|B|\le2$, and so $|B|=2$. By
\Cref{prop:two-fibres-full}, there exist $x,g\in G$ such that
\[
    A=(x+H)\cup(x+g+H).
\]
Then
\[
    D=A-A=(-g+H)\cup H\cup(g+H).
\]
If the image of $g$ in $G/H$ had order $1$, the two cosets in $A$ would
coincide. If it had order $2$, then $G/H$ would contain two-torsion. If it
had order $3$, then $(-g+H)\cup H\cup(g+H)$ would be the subgroup
generated by $g+H$, contrary to the present non-subgroup case. Therefore
the order is different from $1$, $2$, and $3$, the three cosets in
\eqref{eq:D-three-cosets} are distinct, and $D$ is not a subgroup.
\end{proof}

\begin{remark}
Conversely, if $H\le G$ is finite and
\[
    A=(x+H)\cup(x+g+H),
\]
where the image of $g$ in $G/H$ has order at least four, then
\[
    A-A=(-g+H)\cup H\cup(g+H),
\]
with three distinct cosets. Moreover, differences in $H$ have $|A|$
representations, while differences in $g+H$ and $-g+H$ have $|A|/2$
representations. Hence the non-subgroup alternative in \Cref{thm:main} is
attained exactly at the half threshold.
\end{remark}

\section{Counterexamples}\label{sec:counterexamples-remarks}

The two-torsion-free hypothesis in \Cref{thm:main} is essential. The
following construction gives an infinite family of counterexamples in characteristic two.

\begin{theorem}\label{thm:quadratic}
Let $r\ge1$, let $V=\F_2^{2r}$, and write an element of $V$ as
$(x_1,\ldots,x_r,y_1,\ldots,y_r)$. Define
\[
    q(x_1,\ldots,x_r,y_1,\ldots,y_r)=\sum_{i=1}^r x_i y_i\in\F_2.
\]
Let $G=V\times\F_2$ and
\[
    A=\{(x,q(x)):x\in V\}\subseteq G.
\]
Then $|A|=2^{2r}$ and
\[
    A-A=G\setminus\{(0,1)\}.
\]
Moreover,
\[
    r_A(0,0)=|A|,
    \qquad
    r_A(d)=\frac{|A|}{2}
    \quad\text{for every }d\in (A-A)\setminus\{(0,0)\}.
\]
Consequently \eqref{eq:half-popular} holds. However, $A-A$ is neither a
subgroup of $G$ nor a union of at most three cosets of any common subgroup
of $G$.
\end{theorem}

\begin{proof}
Since $G$ has characteristic two, subtraction and addition coincide. Write
\[
    u=(\alpha_1,\ldots,\alpha_r,\beta_1,\ldots,\beta_r)\in V,
    \qquad
    x=(x_1,\ldots,x_r,y_1,\ldots,y_r)\in V.
\]
The derivative of $q$ in the direction $u$ is
\[
\begin{aligned}
    q(x+u)+q(x)
    &=\sum_{i=1}^r (x_i+\alpha_i)(y_i+\beta_i)+\sum_{i=1}^r x_i y_i  \\
    &=\sum_{i=1}^r \alpha_i y_i+
      \sum_{i=1}^r \beta_i x_i+
      \sum_{i=1}^r \alpha_i\beta_i .
\end{aligned}
\]
If $u=0$, this derivative is identically zero. Hence
$r_A(0,0)=|V|=|A|$ and $r_A(0,1)=0$.

Assume now that $u\ne0$. The linear part
\[
    L_u(x)=\sum_{i=1}^r \alpha_i y_i+
           \sum_{i=1}^r \beta_i x_i
\]
is a non-zero linear functional on $V$. Therefore $L_u$ is balanced: it
takes each value in $\F_2$ exactly $2^{2r-1}$ times. Adding the constant
$\sum_i\alpha_i\beta_i$ preserves balancedness. Thus, for each
$\varepsilon\in\F_2$,
\[
    |\{x\in V:q(x+u)+q(x)=\varepsilon\}|=2^{2r-1}=|A|/2.
\]
Equivalently,
\[
    r_A(u,\varepsilon)=|A|/2
    \qquad (u\ne0,\ \varepsilon\in\F_2).
\]
It follows that the only element of $G$ not represented as a difference of
two elements of $A$ is $(0,1)$, and the asserted representation counts
follow.

It remains to check the coset conclusion. We have
$|A-A|=|G|-1=2^{2r+1}-1$. This number is not a power of two, so $A-A$ is
not a subgroup of the elementary two-group $G$. If $A-A$ were a union of
$m\le3$ cosets of a common subgroup $K\le G$, then, after deleting
repeated cosets, its cardinality would be $m|K|$, where $m\in\{1,2,3\}$
and $|K|$ is a power of two. Since $2^{2r+1}-1$ is odd and larger than
$3$, no such equality is possible. Hence $A-A$ is not a union of at most
three cosets of a common subgroup.
\end{proof}

\begin{example}\label{ex:smallest}
For $r=1$, the construction gives a four-point subset of $\F_2^3$. After
an invertible affine change of coordinates, it is
\[
    A=\{0,e_1,e_2,e_3\}.
\]
Then
\[
    A-A=\F_2^3\setminus\{e_1+e_2+e_3\},
\]
and every non-zero element of $A-A$ has exactly two representations.
\end{example}

\begin{figure}[t]
\centering
\begin{tikzpicture}[scale=0.95]
\coordinate (000) at (0,0);
\coordinate (100) at (2.2,0);
\coordinate (010) at (0,2.2);
\coordinate (110) at (2.2,2.2);
\coordinate (001) at (0.75,0.75);
\coordinate (101) at (2.95,0.75);
\coordinate (011) at (0.75,2.95);
\coordinate (111) at (2.95,2.95);
\foreach \a/\b in {000/100,000/010,100/110,010/110,001/101,001/011,101/111,011/111,000/001,100/101,010/011,110/111}
    \draw (\a) -- (\b);
\foreach \p/\lab in {000/$0$,100/$e_1$,010/$e_2$,001/$e_3$,110/$e_1+e_2$,101/$e_1+e_3$,011/$e_2+e_3$}
    \fill (\p) circle (2.2pt) node[below right=2pt] {\scriptsize \lab};
\draw[thick] (111) circle (3pt) node[above right=2pt] {\scriptsize $e_1+e_2+e_3$};
\node[align=center] at (1.45,-0.75) {filled points: $A-A$; open point: omitted difference};
\end{tikzpicture}
\caption{The smallest two-torsion obstruction in $\F_2^3$. The difference set is the
cube with one vertex removed, which cannot be a subgroup or three cosets of a common
subgroup.}
\label{fig:cube-counterexample}
\end{figure}
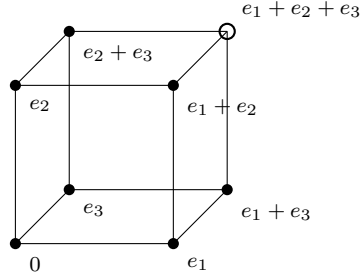

\begin{corollary}\label{cor:lifting}
Let $H$ be a finite abelian group and let $B\subseteq G_0=\F_2^{2r+1}$ be
one of the sets constructed in \Cref{thm:quadratic}. Put
$G=H\times G_0$ and $A=H\times B$. Then
\[
    r_A(d)\ge \frac{|A|}{2}\qquad\text{for every }d\in A-A,
\]
but $A-A$ is neither a subgroup nor a union of at most three cosets of a
common subgroup of $G$.
\end{corollary}

\begin{proof}
For $d=(h,g)\in H\times G_0$, the representation function factors as
\[
    r_A(h,g)=|H|\,r_B(g).
\]
Thus the popularity assertion follows from \Cref{thm:quadratic}. Also
$A-A=H\times(B-B)$. If $A-A$ were a union of at most three cosets of a
common subgroup of $G$, then its image under the quotient map
$G\to G/H\cong G_0$ would express $B-B$ as a union of at most three cosets
of a common subgroup of $G_0$, contradicting \Cref{thm:quadratic}.
\end{proof}

\end{document}